\documentclass[10pt,twoside,twocolumn]{trarticle}

\usepackage{amsmath,amssymb,amsfonts,amsfonts,amsthm}
\usepackage{cite}
\usepackage{tr}
\usepackage{pstricks}
\usepackage{graphicx}
\usepackage{psfrag}
\usepackage{times}
\usepackage{color}
\usepackage{dsfont}                 
\usepackage{algorithm}
\usepackage[noend]{algpseudocode}
    
\newtheorem{lemma}{Lemma}
\newtheorem{theorem}{Theorem}
\newtheorembsp{listing}{Quellcode-Listing}
\newfont{\listfnt}{cmtt7 scaled\magstep0}

\newtheorem{remark}{Remark}
\newcommand{\goto}{\rightarrow}
\DeclareMathOperator{\E}{\mathsf{E}}
\DeclareMathOperator{\Prob}{\mathsf{P}}

\DeclareMathOperator{\argmin}{argmin}

\DeclareMathOperator{\f}{f}

\tr{  
Technical Report TR-LSR-2012-03-14\\[0.1em]
      Institute of Automatic Control Engineering\\[0.1em]
      Technische Universit\"at M\"unchen\\[0.1em]
      March 2012}

\title{{\bf Event-Triggered Estimation of Linear Systems: An Iterative Algorithm and Optimality Properties \footnote{This is an extended version of the paper 'An Iterative Algorithm for Optimal Event-Triggered Estimation' that appears in the proceedings of the ADHS 2012.}}}
\author{ {\large Adam Molin\quad Sandra Hirche} \\
\ \\
Institute of Automatic Control Engineering\\
D-80290 Munich, Germany \\
\texttt{ fax: +49-89-289-28340} \\
\texttt{e-mail: adam.molin@tum.de, hirche@tum.de} \\
\texttt{http://www.lsr.ei.tum.de } \\
}
\date{}
\abs{
This report investigates the optimal design of event-triggered estimation for first-order linear stochastic systems.
The problem is posed as a two-player team problem with a partially nested information pattern. The two players are given by an estimator and an event-trigger. The event-trigger has full state information and decides, whether the estimator shall obtain the current state information by transmitting it through a resource constrained channel.
The objective is to find an optimal trade-off between the mean squared estimation error and the expected transmission rate.
The proposed iterative algorithm alternates between optimizing one player while fixing the other player. 
It is shown that the solution of the algorithm converges to a linear predictor and a symmetric threshold policy, if the densities of the initial state and the noise variables are even and radially decreasing functions. This is achieved by considering the iterative algorithm as a dynamical system and apply Lyapunov methods to show that it is globally asymptotically stable. The effectiveness of the approach is illustrated on a numerical example. In case of a multimodal distribution of the noise variables a significant performance improvement can be achieved compared to a separate design that assumes a linear prediction and a symmetric threshold policy.
}

\begin{document}

  \renewcommand{\baselinestretch}{1.0}
  \maketitle

\section{INTRODUCTION}

In contrast to periodic estimation, where measurements are sampled within equidistant time-intervals, an event-triggered estimator  receives measurement updates in an asynchronous fashion. Event-triggered sampling is also referred to as adaptive sampling in~\mbox{\cite{rabi:siam2012}}, Lebesgue sampling in~\mbox{\cite{AstromBernhardsson02_lebesgue}} and dead-band control in~\cite{otanez2002}, \mbox{\cite{hirchenetwork}}.
The event-trigger is a preprocessing unit situated at the sensor which decides upon its available information, whether to update the estimator with current information. Event-triggered sampling schemes for estimation are very promising in the context of networked control systems, where estimator and plant are spatially distributed and communication is a sparse resource. Examples for such networked control systems are given by sensor networks,  multi-robot systems and distributed power generation networks.
The work in \cite{AstromBernhardsson02_lebesgue} and \cite{rabi:siam2012} showed that event-triggered sampling outperforms periodic sampling with respect to the state estimation error of a first-order linear system in the presence of two different communication constraints. In  \cite{rabi:siam2012}, the communication constraint is induced by limiting the number of transmissions during a finite interval, whereas the work in \cite{AstromBernhardsson02_lebesgue} limits the average transmission rate. Differing to these approaches, we extend the standard minimum mean square estimator problem by an additional communication penalty to reflect the communication constraint in the optimization problem. A similar problem is also studied in \cite{Xu2004_optimal} and \cite{lipsa2011}.

Opposed to the aforementioned work which either fixes the estimator, such as \cite{AstromBernhardsson02_lebesgue,rabi:siam2012,Xu2004_optimal} or computes the estimator from the choice of the event-trigger, such as  \cite{lipsa2011}, we aim at the joint optimal design of the estimator and the event-trigger. Therefore, we formulate a two-player team problem with a nested information pattern, where the players are given by the event-trigger and the estimator. The joint design is motivated by the fact that the choice of the event-trigger may significantly influence the form of the optimal estimator.

The contribution of this report is two-fold. First, it develops an iterative method for the joint design of event-trigger and estimator for first-order stochastic systems with arbitrary distributions. The algorithm iteratively alternates between optimizing one player while fixing the other player. Similar iterative procedures are shown to be very promising methods for calculating optimal policies for team problems with non-classical information patterns, as studied by \cite{karlsson2011} for the Witsenhausen counterexample or by \cite{hajek2008paging} for the joint optimization of paging and registration policies.
It turns out that the proposed iterative method can yield a remarkable decrease of the overall cost compared to a design where the estimator is designed independently of the event-trigger. In such independent design, the optimal estimator takes the form of a linear predictor that assumes that transmission instants are statistically independent of the state, whereas the optimal event-trigger is a even threshold function of the estimation error. In the following, even and symmetric refer to the same meaning.\\
Second, it is shown that the solution of the algorithm converges to the independent design, when the densities of the initial state and the noise variables are symmetric and unimodal functions. This result coincides with results obtained in \cite{lipsa2011}, which uses majorization theory and rearrangement inequalities to show that there always exists a symmetric threshold policiy that outperforms an arbitrary event-triggering law. In fact, we show that symmetric threshold policies are optimal by analyzing the asymptotic behavior of the proposed iterative procedure.Therefore, our approach can be viewed as an alternative line of proof to show that symmetric policies are optimal under the aforementioned assumptions. 
On the other hand, it turns out that symmetry of the densities is not sufficient to show that the separate design is optimal. Numerical simulations indicate significant improvements, when noise densities are symmetric but multimodal.
In fact, simulations indicate a substantial improvement of our approach compared to an independent design in case of a bimodal noise distributions.


The remainder of this report is organized into four sections. In section~\ref{S:statement}, we introduce the stochastic system model and describe the problem setting. Section~\ref{S:main} contains the main results of this report and studies the joint design of event-trigger and estimator.
In section~\ref{S:numerical}, numerical simulations are conducted to validate the proposed method.

\textbf{Notation.} The expectation operator is denoted by~$\E_f[\cdot]$ and the conditional expectation is denoted by~$\E_f[\cdot|\cdot]$, where the underlying  probability measure $\Prob_f$ is parameterized by the policy $f$. The variable $X^k$ denotes the sequence of variables $[x_0,\ldots,x_k]$ and $X_k^l$ denotes the sequence $[x_k,\ldots,x_l]$. The indicator function is denoted by $\mathds{1}_{A}(x)$ taking a value of $1$ if $x\in A$ and $0$ otherwise. 
The complement of a set $A$ is denoted by $A^\text{c}$.
The maximum norm of a vector~\mbox{$x\in\mathbb{R}^n$} is denoted by~$|x|_\infty$. The convolution of two real-valued function $f$ and $g$ is denoted by $f * g$.


\section{PROBLEM FORMULATION}\label{S:statement}

We consider the following stochastic scalar discrete-time process $\mathcal{P}$ driven by noise $w_k$
\begin{equation}
    x_{k+1} = ax_k + w_k,
    \label{E:plant}
  \end{equation}
where \mbox{$a\in\mathbb{R}-\{0\}$}. The system noise $w_k$ takes values in~$\mathbb{R}$ and is an  i.i.d. (independent identically distributed) random variable described by the probability density function~$\phi_w$, which is zero-mean and has finite variance. The initial state, $x_0$ is statistically independent of $w_k$ and is described by density function $\phi_{x_0}$, which has a mean~$\bar{x}_0$ and a finite variance.
System parameters and statistics are known to the event-trigger and estimator.

The system model is illustrated in Fig. \ref{fig:systemmodel}. The process $\mathcal{P}$ outputs the state $x_k$. The event-trigger $\mathcal{E}$ decides upon its available information, whether or not to transmit the current state to the remote state estimator $\mathcal{S}$.
We define the output of the event-trigger as
\begin{equation*}
   \delta_k=\begin{cases} 1 & \text{update $x_k$ sent} \\ 0 & \text{otherwise} \end{cases}
  \end{equation*}
The channel $\mathcal{N}$ can be viewed as a $\delta_k$-controlled erasure channel whose outputs are described by
\begin{equation}
    z_{k} = \begin{cases} x_k & \delta_k=1 \\ \emptyset & \delta_k=0 \end{cases}
    \label{E:measurement}
  \end{equation}
where $\emptyset$ is the erasure symbol. As it will be useful for subsequent analysis, we define the last update time $\tau_k$ as
\begin{align}\label{eq:tau}
\tau_k=\max\{\kappa|\delta_\kappa=1,\ \kappa<k\}
\end{align}
with $\tau_k={-1}$, if no transmissions have occurred prior to~$k$. The variable $\tau_k$ can be described by the following~\mbox{$\delta_k$-controlled} difference equation
\begin{align}\label{eq:tauev}
\tau_{k+1}=\begin{cases}
k & \delta_k=1\\
\tau_k & \delta_k=0
\end{cases}\quad\tau_0=-1.
\end{align}

Admissible event-triggers are given by mappings of their past history to 
\[
\delta_k=f_k(X^k),\quad k=0,\ldots,N-1.
\]
The state estimator $\mathcal{S}$ outputs the state estimate $\hat x_k$ and is given by mappings $g_k$ defined by
\[
\hat x_k = g_k(Z^k),\quad k=0,\ldots,N-1.
\]

The design objective is to jointly design the event-trigger $f=[f_0,\ldots,f_{N-1}]$ and the estimator $g=[g_0,\ldots,g_{N-1}]$ that minimize cost $J$.

\begin{equation}\label{E:opt1}
 J=\E_{f,g}\left[\sum_{k=0}^{N-1} |x_k-\hat x_k|^2 + \lambda \delta_k\right].
\end{equation}
The per-stage cost of $J$ is composed of the squared estimation error $|x_k-\hat x_k|^2_2$ and a communication penalty~\mbox{$\lambda\delta_k$}.
The weight $\lambda$ determines the amount of penalizing transmissions over the channel $\mathcal{N}$.

\begin{figure}[htb]
 \centering
 \psfrag{PLANT}[c][c]{$\mathcal{P}$}
\psfrag{CONTR}[c][c]{$\mathcal{S}$}
\psfrag{N}[c][c]{$\mathcal{N}$}
\psfrag{S}[c][c]{$\mathcal{E}$}
\psfrag{z}[c][c]{$z_{k}$}
\psfrag{x}[c][c]{$x_k$}
\psfrag{y}[c][c]{\tiny } 
\psfrag{u}[c][c]{$\hat x_k$}
\psfrag{d}[c][c]{$\delta_k$} 
\includegraphics[width=0.45\textwidth]{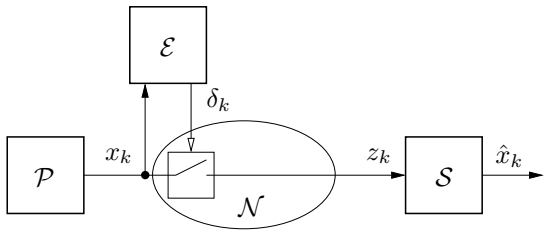}
 \caption{System model of the networked control system with plant $\mathcal{P}$, event-trigger $\mathcal{E}$, state estimator $\mathcal{S}$ and communication channel $\mathcal{N}$.}
 \label{fig:systemmodel}
\end{figure}

\section{JOINT DESIGN OF EVENT-TRIGGER AND ESTIMATOR}\label{S:main}

\subsection{Preliminaries}
We begin with a characterization of the optimal estimator.

\begin{lemma}
For any event-trigger $f$, the optimal state estimator $g^*$ is given by the least squares estimator
\begin{align*}
\hat x_k= g_k^*(Z^k)=\E_f[x_k|Z^k],\quad k=0,\ldots,N-1.
\end{align*}
\end{lemma}
\begin{proof}
Fix an arbitrary event-trigger $f$. The communication penalty term $\E_{f}\left[\sum_{k=0}^{N-1}\lambda \delta_k\right]$ is then constant and can be omitted from the optimization. In the remaining estimation problem the mean squared error~\mbox{$\E_f\left[\sum_{k=0}^{N-1}|x_k-\hat x_k|^2\right]$} is to be minimized. The optimal solution for this problem is given by the least squares estimator $\E_f[x_k|Z^k]$, \cite{Bertsekas2005v1}. This completes the proof.
\end{proof}

In the following, we define the linear predictor $\hat x_{k}^{\text{LP}}$ by the following recursion
\begin{align}\label{eq:lp}
\hat x_{k+1}^{\text{LP}}=\begin{cases}
x_{k+1} & \delta_k=1\\
a\hat x_{k}^{\text{LP}} & \delta_k=0
\end{cases}
\end{align}
with $\hat x_{0}^{\text{LP}}=\bar{x}_0$.
\begin{remark}
The linear predictor can be regarded as the optimal estimator, when having no information about the choice of the event-trigger $f$ and assuming that transmission instances are statistically independent of the state evolution. This also implies that the linear predictor is optimal in the case, when transmission instances are selected in advance.
\end{remark}

Similar to \cite{lipsa2011,MoHi2009}, let us rewrite the optimization problem by defining
\begin{align*}
e_k=x_k-a\hat x_{k-1}^{\text{LP}},\quad k=1,\ldots,N-1
\end{align*}
and $e_0=w_{-1}$, where we define $w_{-1}=x_0-\bar x_0$.
The variable~$e_k$ defines our new state to be estimated and follows the recursion
\begin{align}\label{eq:rec}
e_{k+1}=h_k(e_k,\delta_k,w_k)=(1-\delta_k)ae_k+w_k.
\end{align}
Further, we define $\hat e_k$ to be the least squares estimate~$\E[e_k|\tilde Z^k]$, where $\tilde z_k$ is defined accordingly as
\[
 \tilde z_{k} = \begin{cases} e_k & \delta_k=1 \\ \emptyset & \delta_k=0 \end{cases}
\]

The next lemma gives us further insights into the structure of $\hat e_k$.
\begin{lemma}\label{lem:2}
Let the event-trigger $f$ be fixed. Then, the least squares estimate of $e_k$ is given by
\begin{align*}
\hat e_k=\begin{cases}
e_k & \delta_k=1\\
\alpha_k(\tau_k) & \delta_k=0
\end{cases}
\end{align*}
where $\tau_k$ is defined by Eq. (\ref{eq:tau}) and $\alpha_k(\tau_k)$ is defined by
\begin{align}\label{eq:alpha}
\alpha_k(\tau_k)\hspace{-1pt}=\hspace{-1pt}\E_f\hspace{-2pt}\left[\sum_{l=\tau_k}^{k-1} a^{k-l-1}w_l | \delta_{\tau_k+1}=0,\ldots,\delta_{k}=0\right]\hspace{-3pt}.
\end{align}
\end{lemma}

\begin{proof}
Clearly, we have $\hat e_k=e_k$ for $\delta_k=1$, as~\mbox{$e_k\in \tilde Z^k$}. For $\delta_k=0$, $\tau_k$ is a sufficient statistics for $\hat e_k$. The mapping $\alpha_k$ is determined by applying recursively (\ref{eq:rec}) with~\mbox{$e_{\tau_k+1}=w_{\tau_k}$}. This completes the proof.
\end{proof}

The function $\alpha$ in Lemma \ref{lem:2} can be interpreted as a bias term to improve the state estimate by incorporating additional information $\delta_{\tau_k+1}=\cdots=\delta_k=0$ at time $k$. 

Rather than regarding $\alpha$ as a function of $k$ and $\tau_k$, we will interpret $\alpha$ as a vector in $\mathbb{R}^{\frac{1}{2}N(N+1)}$ by reindexing its entries appropriately.

It is straightforward to see that the estimation error~\mbox{$e_k-\hat e_k$} and~$x_k-\hat x_k$ are identical random variables for a fixed event-trigger $f$, as $e_k$ corresponds to a translatory coordinate transformation of $x_k$ shifted by $-a\hat x_{k-1}^{\text{LP}}$ which is known since the sequence $\delta^{k-1}$ is measurable with respect to~$Z^k$.
Therefore, our initial optimization problem with cost function $J$ can be rewritten as
\begin{align}\label{eq:opt2}
\min_f \E_f\left[\sum_{k=0}^{N-1} (1-\delta_k)|e_k-\alpha_k(\tau_k)|^2 + \lambda \delta_k\right].
\end{align}
It can be observed that the running cost reduces to $\lambda$ and is therefore independent of the current $\alpha_k$ in the case~\mbox{$\delta_k=1$}.
Because of the introduction of the state $e_k$, the event-trigger $f$ is given by a mapping from $E^k$ to $\{0,1\}$. Since there always exists a bijection from $X^k$ to $E^k$ given the variables $\delta_0,\ldots,\delta_{k-1}$, this change of variables does not put any restrictions on the further analysis keeping in mind that any policy expressed in $E^k$ can also be written as a function in $X^k$.

\subsection{Iterative procedure}

What prevents a further study of the optimization problem~(\ref{eq:opt2}) is the fact that the value $\alpha_k$ at $\tau_k$ depends on the particular policy~$f$ chosen up to time $k$. Therefore, methods like dynamic programming are not directly applicable to solve~(\ref{eq:opt2}).
In order to overcome this burden, we relax optimization problem~(\ref{eq:opt2}) by considering the variable $\alpha_k$ as a new decision variable being a function of $\tau_k$. Then, the optimization problem is given by
\begin{align}\label{eq:opt3}
\min_{f,\alpha} J
\end{align}
with
\begin{align}\label{eq:cost}
J(f,\alpha)=\E_{f}\left[\sum_{k=0}^{N-1} (1-\delta_k)|e_k-\alpha_k(\tau_k)|^2 + \lambda \delta_k\right].
\end{align}
The optimization problem (\ref{eq:opt3}) enlarges the set of possible solutions compared to optimization problem (\ref{eq:opt2}), because it omits the constraint for $\alpha$ given by (\ref{eq:alpha}). By considering optimization problem (\ref{eq:opt3}), we are able to specify the structure of the optimal event-trigger, which is given by the following lemma.
\begin{lemma}\label{lem:suff}
Let $\alpha$~be~fixed. Then, for all~\mbox{$k\in\{0,\ldots,N-1\}$} the variables $e_k$ and $\tau_k$ are a sufficient statistics for the optimal event-trigger $\f_k$ .
\end{lemma}
\begin{proof}
The evolution of the pair $(e_k, \tau_k)$ can be regarded as a $\delta_k$-controlled Markov process defined by (\ref{eq:tauev}) and (\ref{eq:rec}). The running cost of $J$ at time $k$ is a function of the pair $(e_k, \tau_k)$, input $\delta_k$ and noise $w_k$. By \cite{Bertsekas2005v1}, this problem can be solved by dynamic programming with $(e_k, \tau_k)$ being the state, which is a sufficient statistics of the optimal solution $f_k$. This completes the proof. 
\end{proof}
Lemma \ref{lem:suff} implies that the optimal event-trigger is a function of $e_k$ and $\tau_k$. It can be observed that for a fixed event-trigger~$f$, the optimal map $\alpha$ can be calculated by Eq. (\ref{eq:alpha}). On the other hand, for any fixed map $\alpha$, the optimal event-trigger $f$ can be calculated by dynamic programming. We therefore define the running cost and the Bellman operator as follows
\begin{align*}
&c^{\alpha_k}_k(e_k,\tau_k,\delta_k)=(1-\delta_k)|e_k-\alpha_k(\tau_k)|^2 + \lambda \delta_k
\\
&\mathcal{T}^{\alpha_k}_k J_{k+1}(\cdot)=
\\&=\min_{\delta_k\in\{0,1\}}c^{\alpha_k}_k(\cdot,\delta_k)+\E\left[J_{k+1}(e_{k+1},\tau_{k+1}) |\cdot,\delta_k\right]
\end{align*}
The value function $J_k$ being a function of the augmented state $(e_k,\tau_k)$ is determined by recursive application of the Bellman equation given by \[
J_k=\mathcal{T}^{\alpha^i_k}_k J_{k+1}\]
 with $J_N\equiv 0$, where the argument in the minimization yields the optimal event-trigger~$f$ and we have
 \[
 J(f,\alpha)=\E_{f}[J_0(e_0,-1)].
 \]

This observation motivates us to propose the following iterative procedure sketched in Fig. \ref{fig:iter}, which alternates between optimizing $f$ while fixing policy $\alpha$ and vice versa. Algorithm \ref{al:1} describes the iterative procedure. With slight abuse of notation, we declared $\tau_k$ as a second subscript instead of an argument of $\alpha_k$.

\begin{algorithm}[phbt]
  \begin{algorithmic}[1]
    \Require {$\alpha^0_{k,\tau_k}\in\mathbb{R},\quad k=0,\ldots,N-1,\tau_k=-1,\ldots, k-1$}
    \State $i\gets 0$
            \label{al:1:init}

    \Statex
	\Repeat
    	
    	\State $k=N$, $J_N\equiv 0$
    	\Repeat 
    		\State $k\gets k-1$
      		\State $J_k\gets\mathcal{T}^{\alpha^i_k}_k J_{k+1}$
      		\State \mbox{$f^i_k(e_k,\tau_k)\in\argmin_{\delta_k\in{0,1}} c^{\alpha^i_k}_k(e_k,\tau_k,\delta_k)$}  $\phantom{\hspace{65pt}}+\E\left[ J_{k+1}(e_{k+1},\tau_{k+1})|e_k,\tau_k,\delta_k  \right]$
    	\Until{{$k= 0$}} \label{al:1:target2}
      		 
    	\State $\alpha_{k,\tau_k}^{i+1}\gets \E_{f^i}\left[\sum_{l=\tau_k}^{k-1} a^{k-l-1}w_l | \delta^k_{\tau_k+1}=0\right]$
    	\State  $i \gets i+1$
	\Until{convergence}

  \end{algorithmic}
\caption{Iterative procedure to calculate $(f,\alpha)$}   \label{al:1}
\end{algorithm}

 
\begin{figure}[tb]
 \centering
 {
 \sffamily
 \psfrag{E}[c][c]{\small estimator $\mathcal{S}$}
 \psfrag{T}[c][c]{\small event-trigger $\mathcal{E}$\phantom{1}}
\psfrag{L}[c][c]{\footnotesize least squares estimation}
\psfrag{D}[c][c]{\footnotesize dynamic programming} 
\includegraphics[width=0.40\textwidth]{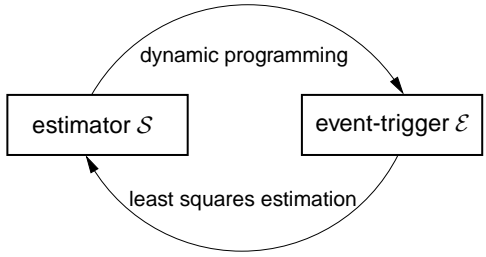}
}
 \caption{Iterative scheme to calculate event-trigger $\mathcal{E}$ and estimator $\mathcal{S}$.}
 \label{fig:iter}
\end{figure}

As the cost $J$ decreases or is at least kept constant in each step of the iteration, the sequence $[(f^0,\alpha^0),(f^1,\alpha^1),\ldots]$ produces a non-increasing succession of costs $J$.

In the following subsection, we are interested in the convergence properties of the proposed iterative algorithm for symmetric unimodal distributions.

\subsection{Symmetric unimodal distributions}

In the following, we consider the iterative procedure described in previous subsection as a discrete-time dynamical system and consider $\alpha$ as the state. By using Lyapunov stability theory we show that $\alpha\equiv 0$ is a globally asymptotically stable equilibrium point, when initial state $e_0$ and the noise process $\{w_k\}$ has a symmetric unimodal distribution.
The next lemma finds a potential equilibrium point only by assuming symmetric distributions.
\begin{lemma}\label{lem:fix}
Let the initial state $e_0$ and the noise process~$\{w_k\}$ have symmetric distributions. Then $\alpha^*\equiv 0$ is a fixpoint of the Algorithm \ref{al:1}. The policy of the event-trigger~$f^*$ that corresponds to $\alpha^*$ is an even mapping of~$e_k$ and independent of $\tau_k$ for~\mbox{$k=0,\ldots,N-1$}. 
\end{lemma}

\begin{proof}
Let us choose the map $\alpha^0$ to be $0$ for all~$k$ and all~$\tau$ in the initialization of Algorithm \ref{al:1}.
The cost function~$J$ reduces then to 
\begin{align*}
J(f,\alpha^0)=\E_f\left[\sum_{k=0}^{N-1} (1-\delta_k)|e_k|^2 + \lambda \delta_k\right]
\end{align*}
where $e_k$ evolves by recursion (\ref{eq:rec}). Therefore, the resulting optimal $f^0_k$ is only a function of $e_k$ for all~\mbox{$k=0,\ldots,N-1$}.
In the following, we first show that the application of the Bellman operator $\mathcal{T}^{0}_k$ preserves symmetry of the value function $J_{k+1}$ for any $k$. Given an even value function~$J_{k+1}$, the conditional expectation $\E\left[J_{k+1}(e_{k+1},\tau_{k+1}) |\cdot,\delta_k\right]$ preserves symmetry for both $\delta_k=0$ and $\delta_k=1$. Adding the cost $c^{0}_k(\cdot,\delta_k)$ also preserves symmetry, because the sum of two even functions is again even. Taking the pointwise minimum of two even functions yields an even function. Therefore, an even function remains even after application of the Bellman operator.
As $J_N\equiv 0$ is an even function, it follows by induction that every value function $J_{k}$ is even for~\mbox{$k\in\{0,\ldots,N-1\}$}. This implies that the  $f^0_k$ resulting in the first iteration step from Algorithm \ref{al:1} is an even mapping of $e_k$, if $\alpha^0\equiv 0$.\\
Next, we calculate $\alpha^1$ assuming $f^0_k$ being an even function of $e_k$ for~\mbox{$k\in\{0,\ldots,N-1\}$}.
Let $\phi_{e_k|\tau}$ be defined as the density function of the conditional probability distribution of $e_k$ given $\tau_k$ and $\delta_k=0$, when using event-trigger $f^0$. The definition of $\phi_{e_k|\tau}$ yields the following calculation of $\alpha_{k,\tau}^{1}$
\begin{align*}
 \alpha_{k,\tau}^{1} = \int_{e\in\mathbb{R}} e \cdot\phi_{e_k|\tau}(e) de
\end{align*}
For $k=0$, $\phi_{e_k|\tau}$ is determined by truncating the density function $\phi_{e_0}$ of the initial state $e_0$ at all $(e,\tau)$, where~$f^1_0$ takes a value of $1$ and by normalizing the resulting function, i.e.
\begin{align}\label{eq:pflem41}
\phi_{e_0|\tau}(e)=\frac{\phi_{e_0}(e)\cdot(1-f^0_0(e,\tau))}{\int_{e\in\mathbb{R}} \phi_{e_0}(e)\cdot(1-f^0_0(e,\tau)) de}.
\end{align}
Since $\phi_{e_0}$ and $f^0_0$ are even functions, we conclude that $\phi_{e_0|\tau}$ is even and therefore we have $\alpha_{0,-1}^{1}=0$.
Along the same lines, we can show  that $\phi_{e_{k}|k-1}$ is even and $\alpha_{k,k-1}^{1}=0$ for~\mbox{$k\in\{1,\ldots,N-1\}$} by replacing $\phi_{e_0}$ with $\phi_{w}$ in (\ref{eq:pflem41}) .
For a constant $\tau$, the conditional density function $\phi_{e_{k}|\tau}$ evolves by the recursion
\begin{align*}
&\phi_{e_{k+1}|\tau}(e)=&
\\&=\frac{(\frac{1}{|a|}\phi_{e_k|\tau}(\frac{(\cdot)}{a})*\phi_{w})(e)\cdot(1-f^0_k(e,\tau))}{\int_{e\in\mathbb{R}} (\frac{1}{|a|}\phi_{e_k|\tau}(\frac{(\cdot)}{a})*\phi_{w})(e)\cdot(1-f^0_k(e,\tau)) de}.&
\end{align*}
It can be observed that this recursion preserves symmetry of the conditional density function $\phi_{e_{k}|\tau}$, as $f^0_k$ is an even function. 
Therefore, we have shown that $\alpha^*\equiv 0$ is a fixpoint of Algorithm \ref{al:1}, which completes the proof.



\end{proof}

In above lemma, the distributions need not to be unimodal, but only symmetry properties are required. A natural question arising  from  Lemma \ref{lem:fix} is whether the fixpoint at $0$ is a stable and unique fixpoint. This question is partly answered in the following Theorem by adding the assumption that the distributions are unimodal.

\begin{theorem}\label{thm:1}
Let  the initial state $e_0$ and the noise process~$\{w_k\}$ have symmetric and unimodal distributions. Then,~$\alpha^*\equiv0$ is a globally asymptotically stable fixpoint of Algorithm~\ref{al:1}.
\end{theorem}

The proof can be found in the appendix.

As the iterative Algorithm \ref{al:1} produces a sequence of pairs $(f^i,\alpha^i)$ whose costs are non-increasing with increasing $i$, we conclude that $0$ is the optimal choice for~$\alpha$, when noise distributions are symmetric and unimodal according to Theorem \ref{thm:1}.
The optimal state estimator of $x_k$ is then given by the linear predictor in (\ref{eq:lp}) and is therefore independent of the choice of the event-trigger~$f$. The distribution of the initial state $x_0$ must be also symmetric and unimodal, but its mean $\bar x_0$ can be chosen arbitrarily. Hence, the symmetry axis of the distribution of $x_0$ need not to be at zero. In order to determine the optimal $f^*$, dynamic programming must only be applied once with~\mbox{$\alpha\equiv 0$}.
 Therefore, the joint design approach in the case of symmetric densities can be considered as an independent design of event-trigger and estimator.\\
This result is in accordance with \cite{lipsa2011} and constitutes an alternative way by analyzing the asymptotic behavior of Algorithm \ref{al:1} to prove that symmetric event-triggering laws are optimal in the presence of symmetric unimodal distributions. Moreover, the iterative algorithm may be applied to arbitrary distributions. Although~\mbox{$\alpha\equiv 0$} is a fix point of the Algorithm~\ref{al:1} by Lemma~\ref{lem:fix} assuming symmetric density functions, the next section shows that an independent approach given by~\mbox{$\alpha\equiv 0$} can be outperformed by Algorithm \ref{al:1} by almost $50\%$. Hence, we can conclude that symmetry of the densities is not sufficient to show that the independent design is optimal. Therefore, additional assumptions are required to show that the independent design is optimal. In the case of Theorem \ref{thm:1} such requirement is given by the unimodality assumption of the density functions.


\section{NUMERICAL VALIDATION}\label{S:numerical}

This subsection intends to outline the benefits of the proposed iterative algorithm by numerical examples. Besides, it validates the obtained results for unimodal noise distributions. We compare the iterative algorithm with the optimal symmetric event-trigger having a linear predictor, i.e. assuming $\alpha\equiv 0$.
Suppose the process defined by~(\ref{E:plant}) with $a=1$, a communication penalty $\lambda=0.5$ and the distribution of the initial state and the system noise are identical and defined by the density function $\phi_w$
\begin{align*}
\phi_w(\mu,\sigma)=\frac{1}{2}\phi_\mathcal{N}(\mu, \sigma)+\frac{1}{2}\phi_\mathcal{N}(-\mu, \sigma)
\end{align*}
with
\[
\phi_\mathcal{N}(\mu,\sigma)=\frac{1}{\sqrt{2\pi\sigma^2}}e^{-\frac{(x-\mu)^2}{2\sigma^2}}.
\]
In the special case of $\mu=0$, we retrieve the normal distribution. In order to facilitate comparability between different distributions, we choose $\mu\in [0,1)$ and set
\[
\sigma=\sqrt{1-\mu^2}
\]
that yields an identical variance of $1$ for all $\mu\in [0,1)$. In the limit $\mu\goto 1$, the noise process degrades to a Bernoulli process taking discrete values $\{-1,1\}$ with probability $\frac{1}{2}$. Various density functions for different $\mu$ are sketched in Fig. \ref{fig:bimodal}.

\begin{figure}[tb]
 \centering
 {
 \sffamily
 \psfrag{Y}[c][c]{\small density function $\phi_w$}
 \psfrag{X}[c][c]{\small system noise $w$\phantom{1}}
\includegraphics[width=0.43\textwidth]{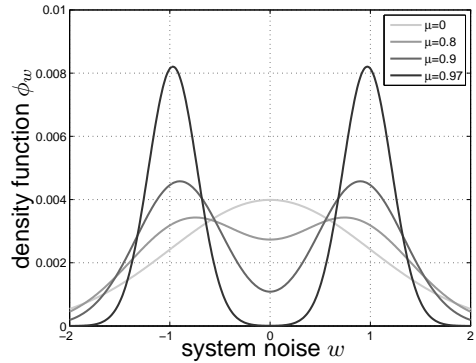}
}
 \caption{Various bimodal/unimodal density functions with zero-mean and identical variance of $1$ composed of two Gaussian kernels shifted by $\pm \mu$.}
 \label{fig:bimodal}
\end{figure}

\begin{figure}[tb]
 \centering
 {
 \sffamily
 \psfrag{Y}[c][c]{\small cost $J$}
 \psfrag{X}[c][c]{\small degree of unimodality $1-\mu$\phantom{1}}
\includegraphics[width=0.43\textwidth]{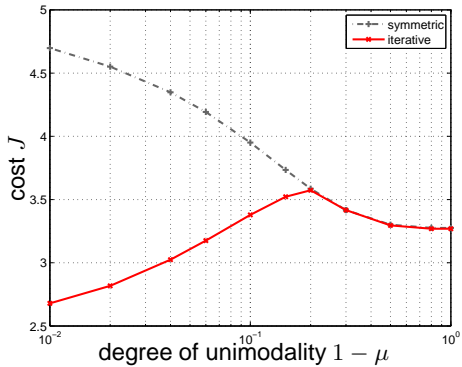}
}
 \caption{Performance comparison for a horizon $N=10$. The degree of unimodality $1-\mu$ ($1$ for zero-mean Gaussian and $0$ for Bernoulli process with discrete parameters in \{-1,1\}) is drawn on a logarithmic scale.}
 \label{fig:perf}
\end{figure}

\begin{figure}[tb]
 \centering
 {
 \sffamily
 \psfrag{Y}[c][c]{\small }
  \psfrag{a}[c][c]{\footnotesize $\alpha_0$}
 \psfrag{X}[c][c]{\small state $x_0$ }
\includegraphics[width=0.43\textwidth]{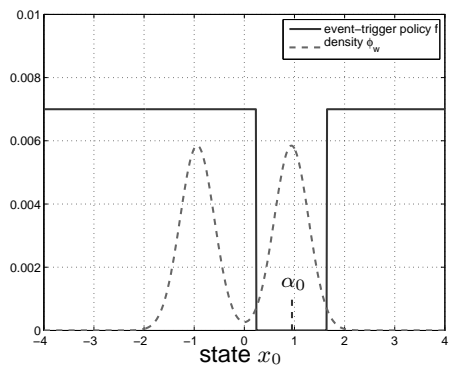}
}
 \caption{Event-trigger policy $f$ (scaled by 0.007) resulting from the iterative Algorithm \ref{al:1} with initial noise distribution $\phi_w$, $\mu=0.95$, horizon $N=1$ and initial choice $\alpha_0^0=0.1$. The algorithm converges to $\alpha_0=0.95$ and an asymmetric event-trigger~\mbox{$f(x_0)=\mathds{1}_{\{[0.25,0.65]\}^\text{c}}(x_0)$}. }
 \label{fig:alphalaw}
\end{figure} 

We observe that for $\mu<0.8$ the peaks of the bimodal density function are less distinctive. Therefore, we can not expect that large gains of the iterative procedure can be attained compared with the optimal symmetric solution for $\mu<0.8$.
A performance comparison of the iterative procedure and the optimal symmetric event-trigger is drawn in Fig. \ref{fig:perf} for a horizon $N=10$ and various $\mu$. The initialization for the iterative procedure is chosen to be $\alpha^0\equiv 0.1$. As expected the costs are almost identical for $\mu\in[0,0.8]$. This also validates Theorem \ref{thm:1}, since $\phi_w$ is unimodal for sufficient small choice of $\mu$.  For $\mu>0.8$ a rapid performance improvement can be observed. In the limit $\mu\goto 1$, the costs are reduced by a factor of~$45\%$ by the iterative procedure compared with the optimal symmetric event-trigger. This may seem surprising, because the cost function as well as the noise distribution are all even functions.
Fig. \ref{fig:alphalaw} gives an illustrative explanation of such significant performance improvement for~\mbox{$N=1$} and~\mbox{$\mu=0.95$}. With an initial value $\alpha_0^0=0.1$, the iterative algorithm converges to $\alpha_0=0.95$ and an asymmetric event-trigger policy $f(x_0)=\mathds{1}_{\{[0.25,1.65]\}^\text{c}}(x_0)$, whereas the optimal symmetric event-trigger is given by $f(x_0)=\mathds{1}_{\{[-0.7,0.7]\}^\text{c}}(x_0)$.
The event-trigger and estimator resulting from the iterative procedure have therefore an implicit agreement, if no state update is sent over the resource-constrained channel. In that case, no transmission indicates the estimator that the state $x_0$ is situated at the right peak resulting in the estimate $\alpha_0$. In contrast to that, the linear predictor defined in (\ref{eq:lp}), which is optimal for the symmetric event-trigger, is independent to the choice of the threshold of the symmetric event-trigger and the noise-distribution.


\section{CONCLUSIONS}

By considering the joint optimal design of state estimator and event-trigger as a two-player problem, we were able to develop an efficient iterative algorithm, which alternate between optimizing the estimator while fixing the event-trigger and vice versa. The iterative method shows special properties in the case of unimodal and symmetric distributions of the uncertainty. In this situation it is shown that the optimal event-triggered estimator can be obtained by a separate design and is given by a linear predictor and a symmetric threshold policy. This result is along previous results and offers an alternative line of proof for showing that such separate design is optimal in case of symmetric unimodal distributions.

In the case of symmetric and bimodal distributions, the iterative procedure offers a systematic method, which leads surprisingly to asymmetric event-triggers and biased estimators that outperform symmetric threshold policies. 

Similar properties of the iterative method are likely to hold as well in the case of multidimensional systems and are a subject of current investigations. Further research also investigates to extend the proposed iterative procedure to a sensor network setting, where various spatially distributed sensors shall find a common state estimate through exchanging information through a common digital network.


\bibliographystyle{ieeetr}
{\small
\bibliography{sigproc.bib}
}

\appendix
\section{Proof of Theorem \ref{thm:1}}    

\begin{proof}
First, we define the following time-variant transformations of $e_k$ and $\alpha_{k,\tau_k}$ by
\begin{align*}
y_k&=\frac{1}{a^k}e_k,\quad k=0,\ldots,N-1,\\
\beta_{k,\tau_k}&=\frac{1}{a^k}\alpha_{k,\tau_k},\quad k=0,\ldots,N-1,\\
&\phantom{\hspace{20mm}}\tau_k=-1,\ldots,k-1.
\end{align*}
By this transformation, the running cost and the Bellman operator are defined by
\begin{align*}
&\hat c^{\beta_k}_k(y_k,\tau_k,\delta_k)=(1-\delta_k)a^{2k}|y_k-\beta_{k,\tau_k}|^2 + \lambda \delta_k
\\
&\hat{\mathcal{T}}^{\beta_k}_k \hat J_{k+1}(\cdot)=\min_{\delta_k\in\{0,1\}}\hat c^{\beta_k}_k(\cdot,\delta_k)+\\ &\phantom{\hspace{20mm}}+\E\left[\hat J_{k+1}(y_{k+1},\tau_{k+1}) |\cdot,\delta_k\right]
\end{align*}
The optimization problem (\ref{eq:opt3}) can then be restated by replacing $J$ with $\hat J$ defined by 
\begin{align*}
\hat J(f,\beta)&=\E_{f}\left[\sum_{k=0}^{N-1} \hat c^{\beta_k}_k(y_k,\tau_k,\delta_k)\right]
\end{align*}
The event-trigger $f_k$ is a function of $y_k$ and $\tau_k$, where $y_k$ evolves by
\begin{align*}
y_{k+1}&=(1-\delta_k)y_k+v_k,\quad y_0=e_0
\end{align*}
with $v_k=\frac{1}{a^k}w_k$ and the evolution of $\tau_k$ is given by~(\ref{eq:tauev}).  It is easy to see that the distribution of $v_k$ is again unimodal and symmetric.
In the following, we adapt Algorithm \ref{al:1} to the transformed system. We consider $\beta^i$ as a vector in $\mathbb{R}^{\frac{1}{2}N(N+1)}$ that evolves by the procedure defined by (\ref{eq:iter2}). 
By this view,  $\beta^i$ is the state of a non-linear time-invariant discrete-time system described by
\begin{align}\label{eq:iter2}
\begin{split}
f^i=\argmin_f \hat J(f,\beta^i)\\
\beta_{k,\tau}^{i+1}=\E_{f^i}\left[\sum_{l=\tau}^{k-1} v_l | \delta_{\tau+1}=0,\ldots,\delta_{k}=0\right]
\end{split}
\end{align}
In order to analyze the asymptotic behavior with increasing $i$, we introduce the following Lyapunov candidate $V(\beta^i)$ defined by
\begin{align*}\label{eq:lyap}
V(\beta^i)=|\beta^i|_\infty.
\end{align*}
In order to show that $V(\beta^i)$ is decreasing with respect to~$i$, we establish several auxiliary results.
For notational convenience, let $\beta^i_\infty$ be defined as
\[
\beta^i_\infty=|\beta^i|_\infty.
\]
What we want to show first is that for every event-trigger~$f^i$ resulting from (\ref{eq:iter2}) for a given $\beta^i$, we have 
\begin{align}
\begin{split}\label{eq:revLAW}
f^i_k(\beta^i_\infty+\Delta,\tau)=0\quad\implies\quad f^i_k(\beta^i_\infty-\Delta,\tau)=0\\
\hfill\forall \Delta\geq 0,k=0,\ldots,N-1,\tau=-1,\ldots,k-1.
\end{split}
\end{align}
The validity of above implication is shown by induction starting with $k=N-1$. We fix a $\beta^i$ and apply dynamic programming to obtain $f^i$. Because of $\hat J_N\equiv 0$, the value function $\hat J_{N-1}$ is then given by
\begin{align*}
\hat J_{N-1}(y,\tau)=\min_{\delta\in\{0,1\}}\hat c^{\beta^i_{N-1}}_{N-1}(y,\tau,\delta)
\end{align*}
Note that the running cost exhibits the symmetry property
\begin{align*}
&\hat c^{\beta^i_k}_k(\beta^i_{k,\tau}+\Delta,\tau,\delta)=\hat c^{\beta^i_k}_k(\beta^i_{k,\tau}-\Delta,\tau,\delta),
\\ &\qquad\qquad\qquad\qquad\forall \Delta\in\mathbb{R},\delta\in\{0,1\}
\end{align*}
with $\tau=-1,\ldots,k-1$ and the monotonicity property
\begin{align*}
&0\leq \Delta_1\leq\Delta_2\\
&\implies \hat c^{\beta^i_k}_k(\beta^i_{k,\tau}+\Delta_1,\tau,\delta) \leq \hat c^{\beta^i_k}_k(\beta^i_{k,\tau}+\Delta_2,\tau,\delta)
\end{align*}
for $\delta\in\{0,1\}$ and $\tau=-1,\ldots,k-1$. Both properties are preserved after taking the minimum over $\delta$ implying that they are also valid for $\hat J_{N-1}$. Therefore, we obtain
\begin{align}\label{eq:revIA}
\hat J_{N-1}(\beta^i_\infty+\Delta,\tau) \geq \hat J_{N-1}(\beta^i_\infty-\Delta,\tau),\quad \forall \Delta\geq 0
\end{align}
with $\tau=-1,\ldots,N-1$. For $\Delta\leq\beta^i_\infty-\beta^i_{k,\tau}$, inequality (\ref{eq:revIA}) is valid due to the monotonicity property of $\hat J_{N-1}$. In case of~\mbox{$\Delta>\beta^i_\infty-\beta^i_{k,\tau}$}, we have
\begin{align*}
&\hat J_{N-1}(\beta^i_\infty-\Delta,\tau)\\
&=\hat J_{N-1}(\beta^i_\infty-\beta^i_{k,\tau}+\beta^i_{k,\tau}-\Delta,\tau)\\
&=\hat J_{N-1}(\beta^i_{k,\tau}+(\beta^i_{k,\tau}-\beta^i_\infty+\Delta,\tau)\\
&\leq \hat J_{N-1}(\beta^i_\infty+\Delta,\tau).
\end{align*}
The second equality is due to the symmetry property and the inequality is due to the monotonicity property as 
\[
\beta^i_{k,\tau}\leq\beta^i_{k,\tau}+(\beta^i_{k,\tau}-\beta^i_\infty+\Delta)\leq\beta^i_\infty+\Delta.
\]
By knowing that the value function $\hat J_{N-1}=\lambda$ is constant for all pairs $(y,\tau)$, when $\delta_{N-1}=1$, we have
\begin{align*}
f^i_{N-1}(\beta^i_\infty-\Delta,\tau)=1\\
\implies \lambda=\hat J_{N-1}(\beta^i_\infty-\Delta,\tau)\leq \hat J_{N-1}(\beta^i_\infty+\Delta,\tau)\\
\implies J_{N-1}(\beta^i_\infty+\Delta,\tau)=\lambda\\
\implies f^i_{N-1}(\beta^i_\infty+\Delta,\tau)=1
\end{align*}
Next, we show that by applying the Bellman operator will preserve the inequality given by (\ref{eq:revIA}). Assume, we have
\begin{align}\label{eq:revIhypo}
\hat J_{k+1}(\beta^i_\infty+\Delta,\tau) \geq \hat J_{k+1}(\beta^i_\infty-\Delta,\tau),\quad \forall \Delta\geq 0
\end{align}
with $\tau=-1,\ldots,k-1$. We want to show statement (\ref{eq:revIhypo}) implies
\begin{align}\label{eq:revIcon}
\hat J_{k}(\beta^i_\infty+\Delta,\tau) \geq \hat J_{k}(\beta^i_\infty-\Delta,\tau),\quad \forall \Delta\geq 0
\end{align}
with $\tau=-1,\ldots,k-1$.
The Bellman equation is
\begin{align*}
\hat J_{k}= \hat{\mathcal{T}}^{\beta^i_k}_k \hat J_{k+1}
\end{align*}
For all pairs $(y,\tau)$, where the argument of the minimization in $\hat{\mathcal{T}}^{\beta^i_k}_k$ yields $\delta_k=1$, $\hat J_{k}$ is constant. This also implies that $\hat J_{k}$ takes its maximum for these pairs. 
In the following, we are interested in outcomes for $\hat J_{k}$ in case of $\delta_k=0$. Along the same lines as for $\hat J_{N-1}$, we obtain for the running cost $\hat c^{\beta^i_k}_k$
\begin{align}\nonumber
&\hat c^{\beta^i_k}_k(\beta^i_\infty+\Delta,\tau,\delta)\geq \hat c^{\beta^i_k}_k(\beta^i_\infty-\Delta,\tau,\delta),\\
&\quad\forall \Delta\in\mathbb{R},\delta\in\{0,1\}\label{eq:revRC}
\end{align}
with $\tau=-1,\ldots,k-1$.
We rewrite $\hat J_{k+1}$ to
\[
\hat J_{k+1} = \hat J^\text{SYM}_{k+1} + \hat J^\text{REM}_{k+1}
\]
with
\begin{align}
\hat J^\text{SYM}_{k+1}(y,\tau)&=\begin{cases}
\hat J_{k+1}(y,\tau) & y \leq \beta^i_\infty\\
\hat J_{k+1}(\beta^i_\infty + (\beta^i_\infty - y),\tau) & y > \beta^i_\infty
\end{cases}\\
\hat J^\text{REM}_{k+1}(y,\tau)&=J_{k+1}(y,\tau)-\hat J^\text{SYM}_{k+1}(y,\tau)
\end{align}
 By the assumption (\ref{eq:revIhypo}), we have 
 \begin{align}\label{eq:revREM}
 \hat J^\text{REM}_{k+1}(y,\tau)\begin{cases}
= 0 & y \leq \beta^i_\infty\\
 \geq 0 &y > \beta^i_\infty
 \end{cases}
\end{align}

\addtolength{\textheight}{-75mm}   

Taking the expectation of $\hat J_{k+1}$ given $\delta_k$, $y_k$ and $\tau_k$, gives either a constant function over $(y_k,\tau_k)$ for $\delta_k=1$ or is given by convolution with the density function of~$v_k$ for~\mbox{$\delta_k=0$} denoted by $\phi$. By assumption the density function $\phi$ is symmetric and unimodal. By linearity of the convolution operator, we follow
\begin{align}\nonumber
&\E\left[\hat J_{k+1} |\cdot,\tau_k,\delta_k=1\right]\\
&=J^\text{SYM}_{k+1}(\cdot,\tau_k) *\phi + J^\text{REM}_{k+1}(\cdot,\tau_k) *\phi\label{eq:revIconv}
\end{align}
For the first term of (\ref{eq:revIconv}), we observe that symmetry is preserved, i.e.,
\begin{align}\nonumber
&(J^\text{SYM}_{k+1}(\cdot,\tau_k) *\phi)(\beta^i_\infty+\Delta)\\
& =(J^\text{SYM}_{k+1}(\cdot,\tau_k) *\phi)(\beta^i_\infty-\Delta)\label{eq:revSYM2}
\end{align}
for $\Delta\in\mathbb{R}$. On the other hand due to (\ref{eq:revREM}) and
\[
\phi(y-(\beta^i_\infty+\Delta)) \geq  \phi(y-(\beta^i_\infty-\Delta)), \Delta\geq 0, y\geq \beta^i_\infty
\]
we have for any $\Delta\geq 0$
\begin{align}\nonumber
&(J^\text{REM}_{k+1}(\cdot,\tau_k) *\phi)(\beta^i_\infty+\Delta)\geq\\ &\geq(J^\text{REM}_{k+1}(\cdot,\tau_k) *\phi)(\beta^i_\infty-\Delta). \label{eq:revREM2}
\end{align}
Summing up the terms and taking the minimum to obtain~$\hat J_{k}$, we obtain statement (\ref{eq:revIcon}) by using (\ref{eq:revRC}), (\ref{eq:revSYM2}) and (\ref{eq:revREM2}). By induction, statement (\ref{eq:revIcon}) is valid for all~\mbox{$k=0,\ldots,N-1$}.
Along the same lines as for $N-1$, we follow  (\ref{eq:revLAW}) from statement (\ref{eq:revIcon}).
Equivalently to (\ref{eq:revLAW}), it can be showed that
\begin{align*}
&f^i(-\beta^i_\infty-\Delta,\tau)=0\\
&\implies\quad f^i(-\beta^i_\infty+\Delta,\tau)=0\\
&\hfill\forall \Delta\geq 0,k=0,\ldots,N-1,\tau=-1,\ldots,k-1.
\end{align*}
Let $\phi^i_{y_k|\tau}$ be defined as the density function of the conditional probability distribution of $y_k$ given $\tau_k$ and $\delta_k=0$, when using event-trigger $f^i$. The definition of $\phi^i_{y_k|\tau}$ yields the following calculation of $\beta_{k,\tau}^{i+1}$
\begin{align*}
 \beta_{k,\tau}^{i+1} = \int_{y\in\mathbb{R}} y \cdot\phi^i_{y_k|\tau}(y) dy
\end{align*}
By assuming an event-trigger $f^i$ that satisfies statement~(\ref{eq:revLAW}), we show inductively that
\begin{align}\label{eq:forDENS}
\begin{split}
\phi^i_{y_k|\tau}(\beta^i_\infty+\Delta)\leq\phi^i_{y_k|\tau}(\beta^i_\infty-\Delta),\quad \forall \Delta\geq 0,\\
\hfill k=0,\ldots,N-1,\tau=-1,\ldots,k-1.
\end{split}
\end{align}
For $k=0$, $\phi^i_{y_k|\tau}$ is calculated by truncating the density function $\phi_{y_0}$ of the initial state $y_0$ at all $(y,\tau)$, where~$f^i_k$ takes a value of $1$ and by normalizing the resulting function, i.e.
\begin{align*}
\phi^i_{y_0|\tau}(y)=\frac{\phi_{y_0}(y)\cdot(1-f^i_0(y,\tau))}{\int_{y\in\mathbb{R}} \phi_{y_0}(y)\cdot(1-f^i_0(y,\tau)) dy}.
\end{align*}
As $\phi_{y_0}$ is an even and unimodal function, we have
\begin{align*}
\phi^i_{y_0|\tau}(\beta^i_\infty+\Delta) \leq \phi^i_{y_0|\tau}(\beta^i_\infty-\Delta),\\
\hfill \Delta\geq 0, f^i_k(\beta^i_\infty+\Delta,\tau)=0.
\end{align*}
For all $(y,\tau)$ with $f^i_k(\beta^i_\infty+\Delta,\tau)=1$, we have 
\[
\phi^i_{y_0|\tau}(\beta^i_\infty+\Delta)=0,
\]
which trivially validates inequality (\ref{eq:forDENS}). Similarly as for~\mbox{$k=0$} and~\mbox{$\tau=-1$}, we can prove the validity of~(\ref{eq:forDENS}) for $k\in\{1,N-1\}$ and $\tau=k-1$ by replacing the density function $\phi_{y_0}$ by the density function $\phi_{v_{k-1}}$ of the noise variable $v_{k-1}$.
By assuming that inequality~(\ref{eq:forDENS}) is satisfied for time step $k$, we prove that (\ref{eq:forDENS}) holds for $k+1$ for an arbitrary $k\in\{0,\ldots,N-2\}$ and fixed~\mbox{$\tau\in\{-1,\ldots,k-1\}$}.
For a fixed $\tau$, the density function $\phi^i_{y_{k}|\tau}(y)$  can be calculated by the recursion
\begin{align}
\phi^i_{y_{k+1}|\tau}(y)\hspace{-2pt}=\hspace{-2pt}\frac{(\phi^i_{y_k|\tau}*\phi_{v_k})(y)\cdot(1-f^i_k(y,\tau))}{\int_{y\in\mathbb{R}} (\phi^i_{y_k|\tau}*\phi_{v_k})(y)\cdot(1\hspace{-2pt}-\hspace{-2pt}f^i_k(y,\tau)) dy}.
\end{align}
As having already been observed for $\hat J_{k+1}$, the convolution of $\phi^i_{y_k|\tau}$ with $\phi_{v_k}$ preserves the inequality (\ref{eq:forDENS}). With the same arguments as for $k=0$, we follow that 
\begin{align*}
\phi^i_{y_k|\tau}(\beta^i_\infty+\Delta) \leq \phi^i_{y_k|\tau}(\beta^i_\infty-\Delta),\\
\hfill \Delta\geq 0
\end{align*}
implies
\begin{align*}
\phi^i_{y_{k+1}|\tau}(\beta^i_\infty+\Delta) \leq \phi^i_{y_{k+1}|\tau}(\beta^i_\infty-\Delta),\\
\hfill \Delta\geq 0,
\end{align*}
which concludes the induction.
Inequality (\ref{eq:forDENS}) implies that $\beta_{k,\tau}^{i+1}\leq \beta^i_\infty$. Similarly, it can be showed that~\mbox{$\beta_{k,\tau}^{i+1}\geq-\beta^i_\infty$}. In fact, it is straight forward to see that the inequalities are strict for all $\beta^i_\infty\neq 0$ and therefore the Lyapunov candidate $V$ decreases with increasing $i$ for all~\mbox{$\beta\neq 0$}. Hence, the iterative procedure defined in~(\ref{eq:iter2}) converges to $0$ for any initial condition of $\beta$. By transforming $\beta$ back into the initial state space system, we can conclude the proof.
\end{proof}

\end{document}